\documentclass{amsart}
\usepackage{amssymb}
\usepackage{amsfonts}
\usepackage{amsmath}
\usepackage[usenames,dvipsnames]{color}
\usepackage{tikz}
\usetikzlibrary{matrix,arrows,decorations.pathmorphing}
\usepackage{amsfonts}
\usepackage{amsthm}
\usepackage{enumerate}
\usepackage{mathrsfs}
\usepackage{bbm} 
\usepackage{draftwatermark}
\SetWatermarkScale{4.5}
\SetWatermarkLightness{1}
\SetWatermarkText{DRAFT$V2-E7$}
\DeclareMathAlphabet{\mathpzc}{OT1}{pzc}{m}{it}
\newtheorem{theorem}{Theorem}[section]

\newtheorem{corollary}[theorem]{Corollary}

\newtheorem{example}[theorem]{Example}

\newtheorem{lemma}[theorem]{Lemma}

\newtheorem{proposition}[theorem]{Proposition}
\newtheorem{remark}[theorem]{Remark}

\theoremstyle{definition}
\newtheorem{definition}[theorem]{Definition}

\newtheoremstyle{named}{}{}{\itshape}{}{\bfseries}{.}{.5em}{#1 \thmnote{#3}}
\theoremstyle{named}

\providecommand{\JOin}{\ensuremath{\bigvee}}

\providecommand{\Z}{\ensuremath{\mathbb{Z}}}
\providecommand{\N}{\ensuremath{\mathbb{N}}}

\providecommand{\weak}{\ensuremath{\text{weak}}}

\providecommand{\pleadsto}{\underset{p}{\leadsto}}

\title{Bounded Topological Speedups}

\author[L.\ Alvin]{Lori Alvin}
\address{Department of Mathematics, Bradley University, 1501 W. Bradley Ave., Peoria, IL 61625}
\email{lalvin@bradley.edu}

\author[D.\ Ash]{Drew D. Ash}
\address{Department of Mathematics and Computer Science, Davidson College, 209 Ridge Rd, Davidson, NC 28035}
\email{drash@davidson.edu}

\author[N.\ Ormes] {Nicholas S. Ormes}
\address{Department of Mathematics, University of Denver, 2280 S. Vine Street, Denver, CO 80208}
\email{normes@du.edu}


\begin{document}
\begin{abstract}

This paper explores the range of bounded speedups in the topological category. Bounded
speedups represent both a strengthening of topological speedups as defined in \cite{Ash} and a 
generalization of powers of a transformation. 
Here we show that bounded speedups preserve the structure of two classical minimal Cantor systems. Specifically, a minimal bounded speedup of an odometer is a conjugate odometer, and a minimal bounded speedup of a primitive substitution is again a primitive substitution, though it is never conjugate to the original substitution system. Further, we give bounds on the topological entropy of bounded speedups, and in special cases we compute the topological entropy of bounded speedups.
\end{abstract}
\maketitle
\section{Introduction}
Given a minimal Cantor system $(X,T)$, a \emph{topological speedup} of $(X,T)$ is any dynamical system topologically conjugate to $S:X\rightarrow X$ where $S$ is a homeomorphism of the form $S(x)=T^{p(x)}(x)$ 
for some $p:X\rightarrow\Z^{+}$. In \cite{Ash}, the second author 
characterized the pairs of minimal Cantor systems $(Y,S)$ and $(X,T)$ where $S$ is a speedup of $T$. 
Here we investigate the more restrictive situation, where $S$ is a speedup of $T$ via a 
uniformly bounded (equivalently a continuous) jump function $p$.
The notion of a speedup is closely tied to the notion of \emph{orbit equivalence}. Two dynamical systems $(X,T)$ and $(Y,S)$ are orbit equivalent if, up to conjugacy, 
every $S$-orbit is equal to a $T$-orbit. 

For both orbit equivalence and speedups, results in the measure-theoretic category preceded topological results. In \cite{Dye}, Dye proved that any two ergodic automorphisms of Lebesgue probability spaces are measurably orbit equivalent. In a similar vein, Arnoux, Ornstein, and Weiss showed that every aperiodic automorphism on a Lebesgue probability space is measurably conjugate to a speedup of any ergodic automorphism \cite{AOW}. More restrictive versions 
of both orbit equivalence and speedups yield more specific results. In particular, 
Belinskaya proved that if orbit equivalence is with an integrable jump function then the systems are flip conjugate \cite{Bel} (i.e., the systems are conjugate or one is conjugate to the inverse of the other). Similarly for speedups, in \cite{Neveu2} Neveu computes the entropy of integrable speedups by proving an extension of Abramov's formula.

In the topological category, the most fundamental results about orbit equivalence concern minimal Cantor systems (homeomorphisms $T:X \to X$ where $X$ is a Cantor space and all $T$-orbits are dense). Giordano, Putnam, and Skau proved that orbit equivalence for these systems is completely characterized by an associated unital ordered dimension group, and moreover
two minimal Cantor systems $(X,T)$ and $(Y,S)$ are orbit equivalent if and only if there is a homeomorphism $f:X \to Y$ which carries the simplex $M(X,T)$ of $T$-invariant Borel measures on $X$ to $M(Y,S)$, the simplex of $S$-invariant Borel measures on $Y$ \cite{GPS}.
As shown in \cite{Ash}, these same invariants are relevant to the characterization of 
pairs of minimal Cantor systems where one is a speedup of the other. 
In particular, one minimal Cantor system $(Y,S)$ is a speedup of another $(X,T)$ if and only if an exhaustive surjection of unital ordered dimension groups exists, or equivalently if there is a homeomorphism from $X$ to $Y$ which 
provides an injection of $M(X,T)$ into $M(Y,S)$. It follows easily that orbit equivalent
minimal systems are speedups of one another; the converse remains an open problem. 

For topological orbit equivalence, natural relations arise from assuming continuity properties of the jump function. Boyle proved that two minimal Cantor systems related by an orbit equivalence with a bounded jump function are flip conjugate, providing a topological analog to Belinskaya's result \cite{Bthesis}. The notion of strong orbit equivalence, where the jump functions may have a single point of discontinuity, turns out to be extremely relevant. Two minimal Cantor systems are strongly orbit equivalent if and only if their associated $C^*$-cross products are isomorphic \cite{GPS}. 

In this paper we take up the study of speedups in the case where the jump function is bounded, i.e. \emph{bounded speedups}. Note that a constant power of a transformation is a bounded speedup, e.g. $(X,T^2)$ is a speedup of $(X,T)$. 
Therefore results about bounded speedups capture powers as well. 

The results in this paper generally demonstrate how invariants of $(X,T)$ such as entropy, the space of invariant measures, and the dimension group can change through a bounded speedup. However, we also show that for two well-known families of minimal Cantor systems, odometers and substitution systems, there is less freedom. A minimal bounded speedup of an odometer must be a conjugate odometer (Theorem \ref{sameodom}). A minimal bounded speedup of a minimal substitution system on a Cantor set must be another substitution system (Theorem \ref{subspeed}). 


The structure of the paper is as follows. In Section $2$ we establish some basic properties of bounded topological speedups. We conclude Section $2$ by bounding, both above and below, the topological entropy of bounded topological speedups. This theorem can be thought of as a topological version of Neveu's entropy theorem, as the bounds are in terms of integrating the jump function against various sets of invariant measures. The results do not follow directly from Neveu because the space of invariant measures for a bounded speedup $S$ may be strictly larger than that of the original $T$. 

In Section $3$ we begin our examination of structural properties preserved by bounded topological speedups with odometers. The main result of this section is showing that a minimal bounded speedup of an odometer is a conjugate odometer. Moreover, we give explicit criterion for not only when one can minimally speedup an odometer, but also give a precise description of the form the jump function must take. 

In Section $4$ we switch our focus to minimal substitution systems and achieve a comparable result, albeit with a noticeable difference. A minimal bounded topological speedup of a minimal substitution system is again a minimal substitution system, however, this new substitution system is never conjugate to the original substitution system (in fact, no speedup of an expansive system can be conjugate to the original). Along the way, we provide examples that show bounded topological speedups are, in fact, a strict generalization of powers of a transformation even in the case of substitution systems. 

\section{Bounded Topological speedups and Topological Entropy}
\subsection{Structure of the jump function}
Although many notions here apply more generally, 
we will focus on topological dynamical systems which are 
\emph{minimal Cantor systems}.
By a minimal Cantor system we mean a pair $(X,T)$ where $X$ is a Cantor space 
and $T:X \to X$ is a homeomorphism where every $T$-orbit is dense. 
\begin{definition}
	Let $(X,T)$ be a minimal Cantor system. 
	A \emph{bounded speedup} of $(X,T)$
	is a homeomorphism of the form $(X,S)$ where 
	 $$
	 S(x)=T^{p(x)}(x),
	 $$
	 for some bounded function $p:X\rightarrow\Z^{+}$, or any system topologically conjugate
	to such an $(X,S)$. We will use the notation $T \leadsto S$ to denote when $S$ is a bounded speedup of $T$, 
	and $T \underset{p}{\leadsto} S$ when $S$ is a bounded speedup of $T$ with jump function $p$. 
\end{definition}
Throughout the paper when we say $S$ is a bounded 
speedup of $T$ we will typically assume (without loss of generality) that 
$S$ is of the form $S(x) = T^{p(x)}(x)$ as opposed to a conjugate version of such a map. 

Note that the definition above does not imply that the system $(X,S)$ is minimal, only aperiodic. If all 
$S$-orbits are dense in $X$, then we will say that $(X,S)$ is a \emph{minimal bounded speedup} of $(X,T)$.

It was shown in \cite{Ash} that the function $p$ associated to any speedup is lower semicontinuous. Below we see that $p$ is bounded if and only if $p$ is continuous. 
\begin{proposition}
	Let $p:X\rightarrow\mathbb{Z}^{+}$ and suppose that $S(x)=T^{p(x)}(x)$ defines a speedup of the minimal Cantor system $(X,T)$, then $p$ is bounded if and only if $p$ is continuous.
\end{proposition}
\begin{proof}
	The converse is clear, and thus we only show the necessary condition. Suppose $p$ is bounded, then
	$$
	p(X)=\{z_{1},\dots,z_{n}\}
	$$
	for some $n\in\Z^{+}$. 
	The sets $p^{-1}(\{z_{i}\})$ form a finite partition of $X$. Therefore, if we
	show that each such set is closed it will follow that each is open as well, completing the proof. 
	
	Let $\{x_n\}_{n=1}^{\infty}$ be a sequence of points in $p^{-1}(\{z_{i}\})$ which 
	converge to a point $x$. 
	Then $$S(x) = \lim_{n \to \infty} S(x_n) = \lim_{n \to \infty} T^{z_i}(x_n) = T^{z_i}(x)$$
	thus $p(x) = z_i$. 
\end{proof}

The \emph{orbit} of a point $x$ for a homeomorphism $T:X \to X$ is the set
$\{ T^j x : j \in \mathbb{Z} \}$. 
We define an \emph{orbit block} of length $n$ with respect to the 
point $x$ and map $T$ to be the following set
$$
\mathcal{O}(T,x,n)=\{x,Tx,\dots,T^{n-1}x\}.
$$
\begin{lemma}
	Let $(X,T)$ be a minimal Cantor system and suppose $S$ is a bounded speedup of $T$. 
	There is a constant $c \in \mathbb{Z}^+$ such that 
	every $T$-orbit is the union of exactly $c$ different $S$-orbits. 
\end{lemma}
\begin{proof}
	Let $c(x)$ denote the cardinality of the set of 
	distinct $S$-orbits that are a subset of the $T$-orbit of $x$.
	
	As the jump function $p$ for $T \pleadsto S$ is bounded, we may set $M= \sup_{x\in X} p(x)$ and fix $x \in X$. 
	Then each $S$-orbit that is a subset of the $T$-orbit of $x$
	must intersect the set $\mathcal{O}(T,x,M+1)$. 
	It follows that the number of distinct $S$-orbits that are a subset of the $T$-orbit of $x$
	is equal to the cardinality of $S^{-1}\mathcal{O}(T,x,M) \setminus \mathcal{O}(T,x,M)$. Therefore, $c$ is finite, so $c:X \to \mathbb{Z}^+$ is a well-defined function. The function $c$ is $T$-invariant by definition. It remains to show that $c$ is continuous. 
	
	Note that value of $c$ can be determined from the knowledge of the values of 
	$p$ at the points $T^{-M}(x), T^{-M+1}(x), \ldots, T^{-1}(x)$.
	Therefore, since $p$ is continuous, $c$ is continuous.  As $c$ is continuous and $T$-invariant, and $T$ is minimal we conclude that $c$ is constant. 
\end{proof}
We call $c$ the \emph{orbit number} for $T\rightsquigarrow S$ if $c$ is the constant from the previous lemma. 

For a topological dynamical system $(X,T)$ we say that a continuous function 
$g:X \to \mathbb{R}$ is a $T$-\emph{coboundary} if there exists another continuous function 
$f:X \to \mathbb{R}$ such that $g = f - f \circ T$.
We will use the following theorem due to Gottschalk and Hedlund to show that the jump function is the orbit number associated to the speedup plus a $T$-coboundary.

\begin{theorem}[Gottschalk $\&$ Hedlund] \label{GH}
	Let $T$ be a minimal transformation of the compact metric space $X$, and $g\in C(X)$. The following are equivalent:
	\begin{enumerate}
		\item $g=f-f\circ T,$ for some $f\in C(X)$
		\item There exists $x_{0}\in X$ for which
		$$
		\sup_{n}\left|\sum_{j=0}^{n-1}g\circ T^{j}(x_{0})\right|<\infty.
		$$
	\end{enumerate}
\end{theorem}

\begin{lemma}\label{Nic Lemma}
	Let $(X,T)$ be a minimal Cantor system. Suppose $T \underset{p}{\leadsto} S$ with orbit 
	number $c$. Then there is an $f \in C(X, \mathbb{Z})$ 
	such that $p(x) = c + f(x) - fT(x)$ for all $x \in X$. 
\end{lemma}
\begin{proof}
	Let $M = \sup_{x \in X} p(x)$. Fix $x_0 \in X$ and $N> 2 M$. 
	We know that the $T$-orbit of $x_0$ is the union of exactly $c$ $S$-orbits
	and that each such $S$-orbit intersects the orbit block $\mathcal{O}(T,x_{0},N)$.
	Let $x_0, x_1, \ldots, x_{c-1}$ be the first elements of the $c$ different 
	$S$-orbits that occur in the sequence $ x_0, T(x_0), \ldots, T^{N-1}(x_0)$. 
	For each $0 \leq k <c$, let $N_k$ be the smallest natural number such that 
	$S^{N_k}(x_k)$ is not in $\mathcal{O}(T,x_{0},N)$.
	Then 
	$$
	\sum_{j=0}^{N-1} pT^j(x) - c N = \sum_{k=0}^{c-1} \sum_{j=0}^{N_k-1} pS^j(x_k) - c N
	= \sum_{k=0}^{c-1} \sum_{j=0}^{N_k-1} \left( pS^j(x_k) - N \right).
	$$
	Note that the sum $s_k = \sum_{j=0}^{N_k-1} pS^j(x_k)$ is exactly the number satisfying
	$T^{s_k}(x_k) = S^{N_k}(x_k)$. Since $x_k = T^j(x_0)$ for $j \in [0,M)$ and 
	$S^{N_k}(x_k) = T^j(x_0)$ for $j \in [N,N+M)$, we see $N-M< s_k < N+M$
	and $s_k-N \in [-M,M]$ for each $k$. 
	Therefore, the sum above is bounded between $-cM$ and $cM$. 
\end{proof}

When $S$ is a bounded speedup of $T$, a key question will be whether the jump function 
is also a constant plus an $S$-coboundary. The following proposition demonstrates the consequences of this condition. 

\begin{proposition} \label{Scob}
Let $(X,T)$ be a minimal Cantor system, and suppose
$T \pleadsto S$ where $S:X \to X$ is minimal and $p(x) = c + gS(x) - g(x)$. 
Then the function $h(x) = T^{-g(x)}(x)$ provides a factor map 
from $(X,S)$ onto the minimal Cantor system $(h(X),T^c)$. Furthermore, 
if $T^c$ is a minimal action of $X$ then $(X,S)$ is conjugate to $(X,T^c)$. 
\end{proposition}
\begin{proof}
Since $g$ is continuous, it is clear that $h(x) = T^{-g(x)}(x)$ is a continuous map 
from $X$ to itself. We have the following relation
\begin{align*}
h(S(x))  & = T^{-gS(x)}(S(x)) \\
& = T^{-gS(x)}T^{p(x)}(x) \\
& = T^{-gS(x)}T^{c+gS(x)-g(x)}(x) \\
& = T^{c-g(x)}(x) \\
& = T^{c}T^{-g(x)}(x) \\
& = T^{c}h(x). 
\end{align*}
It follows from the above that $T^c(h(X)) = h(S(X))=h(X)$ and since 
$T^c$ is a continuous injection, $T^c:h(X) \to h(X)$ is a homeomorphism. 
The $h$-image of an $S$-orbit in $X$ is a $T^c$-orbit in $h(X)$ and 
by continuity, the $h$-image of the closure of an $S$-orbit is 
the closure of a $T^c$-orbit. Since $S$ is minimal, the closure of every 
$T^c$-orbit is dense in $h(X)$. Therefore, $(T^c,h(X))$ is a minimal 
Cantor system. 

For the \lq \lq furthermore\rq \rq \ claim, note that 
if $T^c:X \to X$ is minimal, then each $T^c$-orbit is dense in $X$ 
and therefore $h(X)=X$. To see that $h$ is one-to-one, first note that if $x,y$ 
are in different $T$-orbits then $h(x) \neq h(y)$. If
$y = S^k(x)$ for $k \neq 0$ then $h(y) = h(S^k(x)) = T^{kc}(h(x)) \neq h(x)$. 
So assume $x,y$ are in the same $T$-orbit but separate $S$-orbits. 
Note that each $T$-orbit is the union of $c$ distinct $S$-orbits and also 
the union of $c$ distinct $T^c$-orbits. Since $h(X)=X$, the map $h$ must induce
a bijection between the $T^c$-orbits and the $S$-orbits which are subsets of a single
$T$-orbit. Therefore, $h(x)$ and $h(y)$ must be in separate $T^c$-orbits. 
This completes the proof that $h$ is one-to-one and therefore
$h$ is a conjugacy from $(X,S)$ to $(X,T^c)$. 
\end{proof}

Therefore there are two possible 
distinctions between a bounded speedup $S$ with orbit number $c$ and
the power $T^c$. First, it may be that $T^c$ is not minimal 
on $X$, but there exists a bounded speedup $S$ with orbit number $c$ such that 
$S$ is minimal on $X$. 
Second, while the jump function for a minimal bounded speedup $S$
is always equal to a constant $c$ plus a $T$-coboundary, 
it may not be equal to $c$ plus an $S$-coboundary. 
In Section $4$, we show that both of these possibilities are realized, 
even when restricting to the case where $T$ is a substitution
system. In general, it is possible that both $T^c$ and $S$ are minimal, 
but the jump function has integral 
different from $c$ for some $S$-invariant measure, implying that
there are $S$-invariant Borel probability measures which are 
not $T$-invariant (see Example \ref{exammeas}). 

\subsection{Introducing new invariant measures}

Let $(X,T)$ be a minimal Cantor system and $(Y,S)$ a speedup of $(X,T)$. As shown in \cite{Ash}, 
the speedup relation induces a homeomorphism $\varphi:Y\rightarrow X$ and a resulting 
injection $\varphi_{*}:M(X,T)\hookrightarrow M(Y,S)$. Below we outline the construction of an example that illustrates that this injection need not be a bijection, even for bounded speedups. Our example is one where $(X,T)$ is a minimal Cantor system such that $(X,T^2)$ is also minimal, with $T$ uniquely ergodic, but where $T^2$ has two invariant measures. 

First, let $(Y,S,\nu)$ be an ergodic automorphism of a Lebesgue probability space which is a tower of height two over a mixing system. 
Let $(X,T)$ be a mixing minimal Cantor system which is uniquely ergodic 
(e.g. apply the Jewett-Krieger Theorem to find a minimal model for a mixing ergodic system). 
Let $\mu$ denote the unique $T$-invariant measure on $X$. 
Now, we may apply Theorem 2.5 of \cite{Ormes} to create a minimal Cantor system $(X,T')$
which is strongly orbit equivalent to $(X,T)$ such that $(X,T',\mu)$ is measurably conjugate
to $(Y,S,\nu)$. Strong orbit equivalence is a notion defined in \cite{GPS}, we cite here the relevant properties. 

As a result of this construction,  
\begin{enumerate}
\item $(X,T')$ is uniquely ergodic with invariant measure $\mu$, \label{uniqerg}
\item $(X,(T')^2)$ is minimal, \label{T'min}
\item $(X,(T')^2)$ has two ergodic invariant measures $\nu_1, \nu_2$ and 
$\mu = \frac{1}{2}(\nu_1 + \nu_2)$. \label{measures}
\end{enumerate}

Property \ref{uniqerg} holds because (strong) orbit equivalence preserves spaces of invariant measures \cite{GPS}. 

To see Property \ref{T'min}, first note that $(T')^2$ is minimal if and only if 
there are no continuous functions $f:X \to \mathbb{C} \setminus \{0\}$ 
such that $f\circ T' = - f$, i.e., if
$-1$ is an eigenvalue for $T'$. Strong orbit equivalence does not 
generally preserve eigenvalues, but it does preserve the rational part of the spectrum
-- eigenvalues of the form $\exp (2 \pi i /n)$ where $n \in \mathbb{Z}^+$ 
(see section 2 of \cite{GPS}). Therefore, since $T$ does not have $-1$ as an eigenvalue, neither does $T'$. Thus $(T')^2$ is minimal. 

To see Property \ref{measures}, note that $(X,T',\mu)$ is measurably conjugate to 
$(Y,S,\nu)$. Therefore, 
there is a Borel set $A \subset X$ such that $\mu(A) = \frac{1}{2}$, 
$\mu((T')^{-1}A \cap A) = 0$ and $\mu((T')^{-1}A \cup A) = 1$.
We obtain two new measures $\nu_1$, $\nu_2$ such that 
$\mu = \frac{1}{2} (\nu_1 + \nu_2)$
by setting 
$\nu_1(B) = 2 \mu(A \cap B)$ and $\nu_2(B) = 2 \mu((T')^{-1}A \cap B)$
for any $\mu$-measurable set $B$. Both $\nu_i$ are $(T')^2$ invariant. 

We remark that $\mu$, being the midpoint of two other measures, is not ergodic for $T'$. 
This illustrates that the injection $\varphi_{*}:M(X,T)\hookrightarrow M(X,S)$ need not preserve ergodic measures.

\subsection{Entropy}
We recall the definition of both topological and measure-theoretic 
entropy for minimal Cantor systems.
See \cite{Walters} for much more. 
Let $(X,T)$ be a Cantor system and let $\mathcal{P}$ 
be a finite, nonempty, clopen partition of $X$, 
and let $H(\mathcal{P})$ denote the number of nonempty elements of $\mathcal{P}$. 
The entropy $T$ relative to the clopen partition $\mathcal{P}$ is
$$
h(T,\mathcal{P})=\lim_{n\rightarrow\infty}\frac{1}{n}H\left(\JOin_{i=0}^{n-1}T^{-i}\mathcal{P}\right).
$$
The topological entropy of $T$ is then defined as 
$$
h(T)=\sup_{\mathcal{P}}h(T,\mathcal{P}),
$$
where $\mathcal{P}$ ranges over all clopen partitions of $X$. When $\mu$ is a $T$-invariant
Borel probability measure, we may define the measure-theoretic entropy $h_{\mu}(T)$ in 
a similar way, with 
$$H_{\mu}(T,\mathcal{P}) = - \sum_{A \in \mathcal{P}} \mu(A) \log \mu(A)$$
$h_{\mu}(T,\mathcal{P})=\lim_{n\rightarrow\infty}\frac{1}{n}H_{\mu}\left(\JOin_{i=0}^{n-1}T^{-i}\mathcal{P}\right)$,
and $h_{\mu}(T)=\sup_{\mathcal{P}}h_{\mu}(T,\mathcal{P})$. The Variational Principle states
$$h(T) = \sup_{\mu \in M(X,T)} h_{\mu}(T).$$

We now recall the following theorem of Boyle and Handelman which shows that entropy is not preserved by strong orbit equivalence. 
\begin{theorem}[\cite{BH94}] \label{BHThm}
	Suppose $0<\log(\alpha)<\infty$. There exists a homeomorphism $S$ strongly orbit equivalent to the dyadic adding machine such that $h(S)=\log(\alpha)$.
\end{theorem}

Combining the above result with the characterization of strong orbit equivalence and orbit equivalence in \cite{GPS} and the main result in \cite{Ash}, it follows that for any two 
(strongly) orbit equivalent minimal Cantor systems $(X_1,T_1)$ and $(X_2, T_2)$, each is conjugate to a speedup of the other. Thus the above theorem indicates that 
there is no hope to control the entropy of a general speedup. However, in the bounded case there is more to say and we look to a theorem of Neveu for inspiration. 

\begin{theorem}[\cite{Neveu2}]\label{Nev}
	Suppose $(X,\mathscr{B},\mu, T)$ is an ergodic automorphism and $(X,\mathscr{B}, \mu, S)$ is an aperiodic automorphism of the form 
	$$
	S(x)=T^{p(x)}(x),
	$$
	where $p:X\rightarrow\Z^{+}$. Then $h_{\mu}(S)=(\int{p\,d\mu})h_{\mu}(T)$ whenever $\int{p\,d\mu}$ is finite.
\end{theorem} 

In what follows, we provide upper and lower bounds on the entropy of a bounded speedup of $T$ in terms of $p$ and $h(T)$.
Combining Lemma~\ref{Nic Lemma} and Theorem~\ref{Nev}, we obtain a lower bound for the entropy of a bounded speedup.

\begin{proposition}\label{lower}
Let $(X,T)$ be a minimal Cantor system and $S:X\to X$ a minimal bounded speedup of $(X,T)$. 
Then $h(S)\ge c\cdot h(T)$ where $c$ is the orbit number for $T\underset{p}{\leadsto} S$.
\end{proposition}
\begin{proof}
	Let $c\in\Z^{+}$ be the orbit number for $T\underset{p}{\leadsto} S$ and let $p:X \to \mathbb{Z}^+$ be the jump function. Consider the following calculation:
	\begin{align*}
	h(S)&=\displaystyle\sup_{\mu\in M(X,S)}h_{\mu}(S) \\
	&\ge\displaystyle\sup_{\nu\in M(X,T)}h_{\nu}(S) \\
	&\ge\displaystyle\sup_{\nu\in \partial_{E}(M(X,T))}\left(\int{p\, d\nu}\right)h_{\nu}(T)\\
	&=\displaystyle\sup_{\nu\in \partial_{E}(M(X,T))}\left(\int{(c+(f-f\circ T))}\,d\nu\right)h_{\nu}(T)\text{ by Lemma~\ref{Nic Lemma}}\\
	&=\displaystyle\sup_{\nu\in \partial_{E}(M(X,T))}c\cdot h_{\nu}(T)\\
	&=c\left(\displaystyle\sup_{\nu\in \partial_{E}(M(X,T))}h_{\nu}(T)\right) \\
	&=c\cdot h(T).
	\end{align*}
\end{proof}

Before proceeding with the proof of the upper bound for the topological entropy of bounded speedups we will need the following lemma.
\begin{lemma}\label{GammaLemma}
	Let $(X,T)$ be minimal Cantor system and suppose
	$T \pleadsto S$ where $S:X \to X$ is minimal. 
	For every $\varepsilon>0$ there exists $N\in\N$ such that for all $n\ge N$ and for every $x\in X$
	$$
	\frac{1}{n}\sum_{i=0}^{n-1}p(S^{i}(x))<\left(\sup_{\mu\in M(X,S)}\int{p}\,d\mu\right) +\varepsilon.
	$$
\end{lemma}
\begin{proof}
	Let $$\displaystyle\int{p}\, d\mu_{1}=\sup_{\mu\in M(X,S)}\int{p}\,d\mu$$ and assume the conclusion is false; then there exists an $\varepsilon>0$ and an increasing sequence of positive integers $\{n_{k}\}$ and corresponding sequence of points $\{x_{n_{k}}\}$ which have the property that
	$$
	\frac{1}{n_{k}}\sum_{i=0}^{n_{k}-1}p(S^{i}(x_{n_{k}}))\ge\int{p}\,d\mu_{1}+\varepsilon.
	$$
	Define $$\nu_{n_{k}}=\frac{1}{n_{k}}\sum_{i=0}^{n_{k}-1}\delta_{x_{n_{k}}} \circ S^{-i},$$ where $\delta_{x_{n_{k}}}$ represents the Dirac point-mass measure at $x_{n_{k}}$. As $M(X)$ is compact in the $\weak^{*}$ topology there exists $\nu\in M(X)$ and a subsequence $\{n_{k_{\ell}}\}$ for which
	$$
	\begin{tikzpicture}
	\matrix(m)[matrix of math nodes,
	row sep=2.6em, column sep=2.8em,
	text height=1.5ex, text depth=0.25ex]
	{\nu_{n_{k_{\ell}}} & \nu. \\};
	\path[->,font=\scriptsize,>=angle 90]
	(m-1-1) edge node[auto] {$\weak^{*}$} (m-1-2);
	\end{tikzpicture}
	$$
	The measure $\nu$ is $S$-invariant (see \cite[Theorem $6.9$]{Walters}). Since $p$ is continuous, we can make the following estimate
	\begin{align*}
	\displaystyle\int{p}\,d\nu&=\lim_{\ell\rightarrow\infty}\int{p}\,d\nu_{n_{_{k\ell}}}\\
	&=\displaystyle\lim_{\ell\rightarrow\infty}\frac{1}{n_{k_{\ell}}}\sum_{i=0}^{n_{k_{\ell}}-1}p(S^{i}x_{n_{k_{\ell}}})\\
	&\ge\displaystyle\int{p}\,d\mu_{1}+\varepsilon,
	\end{align*}
	which yields our contradiction as $\int{p}\,d\mu_{1}=\sup_{\mu\in M(X,S)}\int{p}\,d\mu$.
\end{proof}
Now we use Lemma~\ref{GammaLemma} to prove our upper bound on the entropy of a bounded speedup.
\begin{proposition}\label{upper}
	Let $(X,T)$ be minimal Cantor system and suppose
	$T \pleadsto S$ where $S:X \to X$ is minimal. 
	Then $$ h(S)\le\sup_{\mu\in M(X,S)}\left(\int{p}\,d\mu\right)h(T).$$
\end{proposition}
\begin{proof}
Suppose $\mu_1 \in M(X,S)$ satisfies 
$\displaystyle \int{p}\,d\mu_1 = \sup_{\mu\in M(X,S)} \int{p}\,d\mu$.
	Fix $\varepsilon>0$ and let $\alpha$ be a finite clopen partition of $X$ such that $p^{-1}(\Z^{+})\le\alpha$. By Lemma~\ref{GammaLemma} there exists $N\in\N$ such that for every $n\ge N$ and every $x\in X$ we have $$\frac{1}{n}\sum_{i=0}^{n-1}p(S^{i}x)<\int{p}\,d\mu_{1}+\varepsilon.$$ It immediately follows then that for every $n>N$
	$$
	\JOin_{i=0}^{n-1}S^{i}(\alpha)\le\JOin_{i=0}^{\lceil (n-1)(\int{p}\,d\mu_{1}+\varepsilon)\rceil} T^{i}(\alpha).
	$$
	From this we can immediately deduce that
	\begin{align*}
	h(S,\alpha)&\le\lim_{n\rightarrow}\frac{1}{n} H\left(\JOin_{i=0}^{\lceil (n-1)(\int{p}\,d\mu_{1}+\varepsilon)\rceil} T^{i}(\alpha)\right)\\
	&=\lim_{n\rightarrow\infty}\frac{\lceil n(\int{p}\,d\mu_{1}+\varepsilon)\rceil}{n}\cdot\frac{1}{\lceil n(\int{p}\,d\mu_{1}+\varepsilon)\rceil}H\left(\JOin_{i=0}^{\lceil (n-1)(\int{p}\,d\mu_{1}+\varepsilon)\rceil} T^{i}(\alpha)\right)\\
	&=\left(\int{p}\,d\mu_{1}+\varepsilon\right)h(T).
	\end{align*}
	It follows that 
	$$h(S)\le\left(\int{p}\,d\mu_{1}+\varepsilon\right)h(T).$$
	Since $\varepsilon$ was arbitrarily given, we may conclude that
	$$
	h(S)\le\left(\int{p}\,d\mu_{1}\right)h(T).
	$$\end{proof}
Combining Proposition~\ref{lower} and Proposition~\ref{upper} we obtain the main theorem of this section below.

\begin{theorem}\label{entropytheorem}
	Let $(X,T)$ be minimal Cantor system and suppose
	$T \pleadsto S$ where $S:X \to X$ is minimal.
	The entropy of $S$ lies within the following interval
	$$
	c\cdot h(T)\le h(S)\le \left(\int{p}\,d\mu_{1}\right)h(T)
	$$
	where is the $c$ orbit number for $T\pleadsto S$ and $\displaystyle\int{p}\,d\mu_{1}=\displaystyle\sup_{\mu\in M(X,S)}\int{p}\,d\mu.$
\end{theorem}

An immediate corollary of this theorem describes the entropy of $S$ when $M(X,T)=M(X,S)$. Observe that in this case the two systems are orbit equivalent.
\begin{corollary}
	Let $(X,T)$ be a minimal Cantor system and $(X,S)$ a speedup of $(X,T)$ with bounded jump function $p:X\rightarrow\Z^{+}$. If $M(X,T)=M(X,S)$, then
	$$
	h(S)=c\cdot h(T)
	$$
	where $c$ is the orbit number for $T\underset{p}{\leadsto} S$.
\end{corollary}
\begin{corollary}
	Let $(X,T)$ be a minimal Cantor system and let $(X,S)$ be a bounded speedup of $(X,T)$. If $h(T)=0$, then $h(S)=0$.
\end{corollary}
\begin{corollary}
	Let $(X,T)$ be a minimal Cantor system and let $(X,S)$ be a bounded speedup of $(X,T)$. If $h(T)>0$, then $h(S)>0$. Moreover, $h(S)\ge h(T)$ with a strict inequality when the jump function $p\not\equiv 1$.
\end{corollary}
\begin{remark}
	We would like to emphasize two general observations about the entropy of topological speedups. First, it follows from Theorem~\ref{entropytheorem} that the entropy of a bounded speedup can only increase, whereas in the unbounded case entropy can decrease. 
	
	Second, the only instances where a bounded speedup of a minimal Cantor system could possibly be conjugate to the original system is when the original system has entropy $0$ or $\infty$.
\end{remark}

\section{Odometers}
In this section we investigate bounded speedups of odometers, a family of minimal Cantor
systems with zero entropy. We begin with some preliminaries on odometers and then characterize which odometers have minimal bounded topological speedups of a particular orbit number 
(see \cite{BK} for a more thorough introduction to odometers). 

\subsection{Background on Odometers} 

Let $\alpha = \langle \alpha_1,\alpha_2,\alpha_3,\ldots \rangle$ be a sequence of integers with each $\alpha_i \geq 2$.  Denote by $X_\alpha$ the set of all sequences $(a_1,a_2,\ldots)$ such that $0\leq a_i\leq \alpha_i-1$ for each $i\geq 1$.  We apply the metric $d_\alpha$ to $X_\alpha$ by $$d_\alpha\left((x_1,x_2,\ldots),(y_1,y_2,\ldots)\right) = \sum_{t=1}^\infty \frac{\delta(x_i,y_i)}{2^i},$$ where $\delta(x_i,y_i) = 0$ if $x_i = y_i$ and $\delta(x_i,y_i) = 1$ if $x_i\neq y_i$.  

The set $X_{\alpha}$ is the set of $\alpha$-adic numbers with 
addition on $X_\alpha$ defined as follows. 
Set $$(x_1,x_2,\ldots) + (y_1,y_2,\ldots) = (z_1,z_2,\ldots)$$ where $z_1 = (x_1+y_1) \mod \alpha_1$, $r_1=0$ and for each $j\geq 2$, 
$z_j = (x_j +y_j + r_j)\mod \alpha_j$ with $r_j = 0$ if 
$x_{j-1}+y_{j-1} + r_{j-1} < \alpha_{j-1}$ and $r_{j} = 1$ otherwise.  

The map $T_\alpha:X_\alpha \to X_\alpha$, defined by $$T_\alpha\left((x_1,x_2,x_3,\ldots)\right) = (x_1,x_2,x_3,\ldots) + (1,0,0,\ldots),$$ is called the {\em $\alpha$-adic odometer} or {\em $\alpha$-adic adding machine map}. It is straightforward to see that the system $(X_{\alpha},T_{\alpha})$ is a minimal Cantor system. 
We will make use of the following results about odometer systems.

\begin{theorem}\label{BKTHM}\cite{BK,BS} Let $\alpha = \langle \alpha_1,\alpha_2,\ldots\rangle$ be a sequence of integers greater than 1.  Let $m_k = \alpha_1\alpha_2\cdots \alpha_k$ for each $i$.  Let $T:X\rightarrow X$ be a continuous map of a compact topological space $X$.  Then $T$ is topologically conjugate to $T_\alpha$ if and only if there is a sequence of partitions $\{ \mathcal{P}(k):k \geq 1\}$
of $X$ such that the following hold.
	\begin{enumerate}
		\item For each positive integer $k$, the partition 
		$\mathcal{P}(k)$ consists of $m_k$ nonempty, 
		clopen sets which are cyclically 		
			permuted by $T$.
		\item For all $k\geq 1$, $\mathcal{P}(k+1)$ refines $\mathcal{P}(k)$.
		\item The sequence of partitions $\{ \mathcal{P}(k)\}$ separates points. 
	\end{enumerate}
\end{theorem}

From Theorem \ref{BKTHM}, it follows that the partitions $\mathcal{P}(k)$ associated
with an odometer $T=T_{\alpha}$ are of the form 
$$\mathcal{P}(k) = \{ T^i A(k) : 0 \leq i < m_k \}$$
where $\cap_k A(k)$ is a singleton $\{x_0\}$ and 
$T^j A(k+1)  \subset T^i A(k)$ if and only if $j \equiv i \mod m_k$.
The partitions $\mathcal{P}(k)$ are the relevant \emph{Kakutani-Rokhlin partitions} for odometers. 
A Kakutani-Rokhlin partition, or a KR-partition, for a minimal Cantor system $(X,T)$ is a partition of the 
form $\mathcal{P} = \{ T^j A_i : 0 \leq j < l_i, 1 \leq i \leq I \}$ where 
each set $A_i$ is clopen. 
The sets $T^jA_i$ are called the \emph{floors} of the KR-partition and $j$ is the height of the floor $T^jA_i$. 
The set $\{ T^j A_i : 0 \leq j < l_i \}$ is referred to as the \emph{$i$th column} of $\mathcal{P}$
and $l_i$ is the \emph{height} of this column. We will call 
the set $\cup_i A_i$ the \emph{base} of $\mathcal{P}$ and 
the set $\cup_i T^{l_i-1}A_i$ the \emph{top} of $\mathcal{P}$.

Theorem  \ref{BKTHM} states that odometers are characterized by the existence of a generating 
sequence of KR-partitions, each comprised of a single column. In such a situation 
we will suppress the subscripts on the base sets $A(k)=A_1(k)$. 
KR-partitions will also play an important role in our discussion of substitution systems 
in Section \ref{sub}. 

The following provides a means for determining when two odometers are topologically conjugate. 
\begin{lemma}\cite[Corollaries 2.6 and 2.8]{BK}\label{Smeasure}
		Let $\beta = \langle \beta_1, \beta_2, \cdots \rangle$ and $\gamma = \langle \gamma_1,\gamma_2,\ldots \rangle$ be such that $\beta_i,\gamma_i \geq 2$ for all $i\in \N$. Set $\mathcal{S}(T_\beta)$ to be the collection of positive integers $k$ such that for some subset $M$ of $X$, $M$ is $T_\beta^i$ minimal but not $T_\beta^j$-minimal for $j<i$.  Then $T_\beta$ and $T_\gamma$ are topologically conjugate if and only if $\mathcal{S}(T_\beta) = \mathcal{S}(T_\gamma)$.
		
		Moreover, let $M_\beta$ be the function mapping the prime numbers to the extended natural numbers $\{0,1,2,\ldots,\infty\}$ by $M_\beta(p) = \sum_{i=1}^\infty n_i$ where $n_i$ is the power of the prime $p$ in the prime factorization of $\beta_i$. Then $T_\beta$ and $T_\gamma$ are topologically conjugate if and only if $M_\beta = M_\gamma$.
	\end{lemma}
	
Hence, if for each prime number $p$ we have $M_\beta(p) = \infty$, then $T_\beta$ is topologically conjugate to both $T_\alpha$ where $\alpha = \langle 2, 2\cdot 3, 2\cdot 3\cdot 5, \ldots \rangle$ and $T_\gamma$ where $\gamma = \langle 2, 3, 4, 5, \ldots \rangle$.  Additionally, for any given $\alpha$, rearranging the $\alpha_i$ values does not change the conjugacy class of the odometer $T_\alpha$.

	\subsection{Speedups of Odometers}
	
In this section we focus on bounded speedups of odometers. We first show that a minimal bounded speedup of an odometer is a conjugate odometer. We then investigate which odometers have nontrivial minimal 
bounded speedups and conclude by characterizing the jump functions $p$ which 
produce a nontrivial minimal speedup of a given odometer.	

	\begin{theorem}\label{sameodom}
Let $(X,T)=(X_\alpha, T_\alpha)$ be an odometer and suppose $T \pleadsto S$
where $S:X \to X$ is minimal. 
Then $(X,S)$ is topologically conjugate to $(X,T)$.
Moreover, $p(x)-c$ is an $S$-coboundary where $c$ is the orbit
number for $T \pleadsto S$. 
	\end{theorem}
	\begin{proof}
Let $T=T_{\alpha}$ be the odometer with $\alpha = \langle \alpha_1, \alpha_2, \alpha_3, \ldots \rangle$. Consider the sequence of partitions $\{\mathcal{P}(k)\}$ of $X$ as in Theorem \ref{BKTHM} for $T$ and $\alpha$. Because  $\{\mathcal{P}(k) \}$ separates points and $p$ is uniformly continuous, there is an $N$ such that if $k\geq N$ and $x,y$ are in the same element of $\mathcal{P}(k)$ then $p(x)=p(y)$. 

Consider the sets of the form $S^j A(k)$ for $j \geq 1$. Since $p$ is constant on elements of $\mathcal{P}(k)$, each set $S^jA(k)$ must be of the form $T^i A(k)$, another element of $\mathcal{P}(k)$. Because $S$ is a minimal homeomorphism of $X$, it follows that $S$ acts as a cyclic permutation on elements of $\mathcal{P}(k)$. By Theorem \ref{BKTHM}, 
$S$ is topologically conjugate to an odometer
$T_{\beta}$ where $\beta=\langle m_N, \alpha_{N+1}, \alpha_{N+2},\ldots \rangle$. 
By Theorem \ref{Smeasure}, the odometers $T_{\beta}$ and $T_{\alpha}$ are conjugate. 

To see that $p(x)$ is a constant plus $S$-coboundary, 
let $m \geq 1$ be the smallest positive integer such that
$S^m A(k) = A(k)$. Because $S$ acts as a cyclic permutation on a set of $m_k$ elements, 
$m=m_k$. Fix $x \in A(k)$. Then $S^m(x) = T^{p(x,m)}(x)$ where
$$p(x,m) = \sum_{i=0}^{m-1} p(S^ix)$$
Since $T^{p(x,m)}A(k) = S^m A(k) = A(k)$, $m$ must divide $p(x,m)$, i.e., $c m = p(x,m)$ for some $c >0$. Because $p$ is constant on elements of $\mathcal{P}(k)$, 
$c m  = p(z,m)$ for all $z \in A(k)$. Fix $N \geq 1$ and write 
$N = q m + r$ for $q > 0$ and $0 \leq r < m$. Then we have
\begin{align*}
\sum_{j=0}^{N-1} pS^j(x) - c N  
& = q c m + \sum_{j=0}^{r} pS^{qm+j}(x) - cN \\
& = \sum_{j=0}^{r} pS^{qm+j}(x) - c r 
\end{align*}
which is uniformly bounded. Therefore by Theorem \ref{GH}, $p(x)-c$ is an $S$-coboundary. 
\end{proof}
	
\begin{remark} \label{remex}
The above does not imply $(X,S)$ is conjugate
to $(X,T^c)$ since $(X,T^c)$ may not be minimal. An example of this is the odometer
$T_{\alpha}$ with $\alpha = \langle 4,3,3,\ldots \rangle$ and the jump function 
$p$ where $$p(x) = \begin{cases}  
2 & {\rm for }\ x \in A(1) \\
2 & {\rm for }\ x \in TA(1) \\
3 & {\rm for }\  x \in T^2A(1) \\
1 & {\rm for }\  x \in T^3A(1) 
\end{cases}
$$
Here $(X_{\alpha},S)$ is minimal with $c = 2$, 
but $(T_{\alpha})^2$ is not a minimal action of $X_{\alpha}$. 

Moreover, this example illustrates the potential difference between 
$T^c$ and a bounded speedup with orbit number $c$. 
\end{remark}

Next we address the question of which odometers $(X_{\alpha},T_{\alpha})$ 
admit a bounded minimal speedup with a particular orbit number $c$. 
	
\subsection{A necessary condition for minimality}

Let $(X,T)$ be an odometer system and suppose $T \pleadsto S$ where 
$S:X \to X$ is a homeomorphism. We wish to write sufficient conditions
so that $S$ is minimal. 

Note that because the function $p$ is continuous and $\mathcal{P}(k)$ separates points,  
for sufficiently large $k$, $p$ is constant on each element of the KR-partition $\mathcal{P}(k)$. 
Further note that since the heights of the KR-partitions $\mathcal{P}(k)$ go to infinity, 
for sufficiently large $k$, $\sup_{x \in X} p(x)$ is smaller than $m_k$, 
the height of the unique column of $\mathcal{P}(k)$. 

We may introduce a \emph{labeling} $\mathcal{L}_k$ of $\mathcal{P}(k)$ 
based on $S$-paths through $\mathcal{P}(k)$. 
Label the base floor $A(k)$ with $0$. 
Label any other floor $F$ in this column with a $0$ 
if $F = S^jA(k)$ and $\sum_{n=0}^{j-1} pS^n(x) < l(k)$ for $x \in A(k)$. 
Now consider an unlabeled floor $T^{j_1}A(k)$ of minimum height $j_1$. 
Label this floor with a $1$. Label any other floor $F$ in this column with a $1$ if 
$F = S^jA(k)$ and $\sum_{n=0}^{j-1} pS^n(x) < l(k)$ for $x \in T^{j_1} A(k)$.
Now label the minimum height unlabeled floor with a $2$ and continue in this 
manner until all floors in the $i$th column are labeled with a label
$0,1,\ldots , c-1$. 

The labeling has the property that if a floor $F$ is the 
$S$-image of a floor $E$ of lower height then $E$ and $F$ have the same labeling.
It is also the case that if $F$ is a floor of $\mathcal{P}(k)$ labeled $\ell$ 
and is not the floor of maximal height with this property then 
$SF$ is equal to the next higher floor which has label $\ell$. 

The sequence of labelings have the property that if $F$ in $\mathcal{P}(k+1)$ is a
floor with $\mathcal{L}_k$-label $\ell$ and height less than $m_k$ then the
floor $E$ in $\mathcal{P}(k)$ with $E \supset F$ has 
$\mathcal{L}_{k}$-label $\ell$ as well. 

For $k$ large enough so that $p$ is constant on each floor of $\mathcal{P}(k)$ and 
$l(k)>\sup p(x)$, we will define a function 
$\pi^{(k)}: \{0,1,\ldots, c-1\} \to \{0,1,\ldots, c-1\}$.
For $0 \leq \ell < c$, 
define $\pi^{(k)}(\ell)$ to be the label of the floor containing $S(x)$ for all 
$x$ in the floor of maximal height in $\mathcal{P}(k)$
\begin{proposition}
	Each $\pi^{(k)}_i$ is a permutation of the set $\{0,1, \ldots ,c-1\}$.
\end{proposition}
\begin{proof}
	If $\pi^{(k)}$ is not injective then there are two distinct points 
	in the same $T$-orbit with the same $S$-image. But $S$ is a homeomorphism, 
	so this is a contradiction. Therefore, $\pi^{(k)}$ is injective, and therefore
	a permutation. 
\end{proof}

Note further the following relation for the $\alpha$-odometer. 
\begin{lemma} \label{permpower}
$$\pi^{(k+1)} =\underbrace{\pi^{(k)}\pi^{(k)}\cdots \pi^{(k)}}_{\alpha_{k+1}} = (\pi^{(k)})^{\alpha_{k+1}}$$
\end{lemma}
\begin{proof}
Suppose $x$ is in a floor of $\mathcal{P}(k+1)$ with label $\ell$ with minimal height. 
Then $x$ is also in a floor of $\mathcal{P}(k)$ with label $\ell$. 
Then if $n_1$ is the minimal integer such that $\sum_{n=0}^{n_1-1} pS^n(x) \geq l(k)$
$x_1=S^{n_1}(x)$ is in a floor with label $\pi^{(k)}(\ell)$. 
If $n_2$ is the minimal integer such that $\sum_{n=0}^{n_2-1} pS^n(x_1) \geq l(k)$
$x_2=S^{n_2}(x_1)$ is in a floor with label $\pi^{(k)}\pi^{(k)}(\ell)$ and so on. 
This process continues until we have $\sum_{i=1}^{I} n_i \geq l(k+1)$ which occurs exactly 
when $I=\alpha_{k+1}$. 
\end{proof}

Given an odometer $(X_{\alpha},T_{\alpha})$ and 
$T \pleadsto S$, a necessary and sufficient condition for $S:X\to X$ to be minimal is given below. 
\begin{lemma} \label{minsuff}
Suppose $(X_{\alpha},T_{\alpha})$ is a minimal odometer system and 
$T \pleadsto S$ where $S:X_{\alpha} \to X_{\alpha}$ is homemorphism. 
The system $(X_{\alpha},S)$ is minimal if and only if for all sufficiently large $k$, 
$\pi^{(k)}$ is cyclic permutation on $\{0,1,\ldots ,c-1\}$ where 
$c$ is the orbit number for $S$. 
\end{lemma}

This leads to the following two theorems which discuss which odometers admit minimal bounded speedups (with 
a given orbit number). 

\begin{theorem}
If $(X_{\alpha},T_{\alpha})$ is an odometer system with $\alpha = \langle \alpha_1,\alpha_2,\ldots \rangle$ and 
$c \geq 1$ is an integer such that for some $N\in \Z^+$, $\gcd(c,\alpha_i) =1$ 
for all $i\geq N$, then $(X_{\alpha},T_{\alpha})$ has a bounded speedup $S$ with 
orbit number $c$ and $S:X_{\alpha} \to X_{\alpha}$ minimal. 
\end{theorem}
	\begin{proof}
Set $m = \alpha_1\alpha_2 \cdots \alpha_N$ and $g = \gcd(c,m)$. 
Without loss of generality we may assume $m>c$.
By Theorem \ref{Smeasure}, 
$T_{\alpha}$ is topologically conjugate to the odometer
$T_{\beta}$
with $\beta=\langle g, \frac{m}{g}\alpha_{N+1}, \alpha_{N+2}, \alpha_{N+3}, \ldots \rangle$.
It suffices to show that 
$(X,T) = (X_{\beta},T_{\beta})$ has a bounded speedup $S$ with 
orbit number $c$.

Let $\{\mathcal{P}(k)\}$ be the sequence of partitions associated to $(X,T)$.
We will define $p(x)$ to be constant on elements 
$\{T^iA(2) : 0 \leq i < M \}$ of $\mathcal{P}(2)$ where $M = m \alpha_{N+1}$. 
Set 
$$p(x) = \begin{cases} 
c & {\rm for }\ x \in T^iA(2), \ 0 \leq i < M-c, M-c+g \leq i < M \\
c+1 & {\rm for }\ x \in T^iA(2), \ M-c \leq i < M-c+g-1 \\
c-g+1 & {\rm for }\ x \in T^iA(2), \  i = M-c+g-1
\end{cases}$$  
One can check that the permutation $\pi^{(2)}$ gives a cyclic permutation 
of $\{0,1,\ldots,c-1\}$. For $k>2$, Lemma \ref{permpower} and the fact
that $\gcd(c,\alpha_{N+k-1})=1$ imply that $\pi^{(2)}$ also gives a cyclic permutation 
of $\{0,1,\ldots,c-1\}$, which completes the proof in one direction.

For the other, suppose that $\pi^{(k)}$ is a permutation of 
$\{0,1,\ldots, c-1\}$ which is not cyclic. Then it contains a cycle of 
order $< c$. But then for some $x$, the $S$-orbit of $x$ only intersects
sets in $\mathcal{P}(k)$ with $\mathcal{L}_k$-labels in that cycle. 
Therefore, $S$ is not minimal.  
\end{proof}

\begin{theorem}\label{nospeedup}
		If $(X_{\alpha},T_{\alpha})$ is an odometer system with 
		$\alpha = \langle \alpha_1,\alpha_2,\ldots \rangle$ and every prime $p$ divides infinitely many 
		$\alpha_i$, then there is no minimal $S: X_{\alpha} \to X_{\alpha}$ with $T \pleadsto S$
		other than $p\equiv 1$.
	\end{theorem}
	\begin{proof} 
	Consider one of the permutations $\pi^{(k)}$ associated with the speedup
	$S:x \mapsto T^{p(x)}(x)$. If $S$ is a speedup with orbit number $c$ then 
	$\gcd(c,\alpha_n) > 1$ for some $n>k$. By Lemma \ref{permpower}, the permutation 
	$\pi^{(n)}$ cannot be a cyclic permutation of $\{0,1,\ldots, c-1\}$ 
	as it is equal to $(\pi^{(n-1)})^{\alpha_n}$ and $\gcd(c,\alpha_n) > 1$.
	\end{proof}
	
	\begin{remark} We note that it follows from Lemma~\ref{Smeasure} and Theorem~\ref{nospeedup} that the only odometer (up to topological conjugacy) that does not have a nontrivial minimal bounded speedup 
	is the odometer $(X_\alpha, T_\alpha)$ with $\alpha = \langle 2, 3, 4, 5, \ldots \rangle$.
	\end{remark}
	We now conclude our investigation of speedups of odometers by providing the following characterization for when a bounded function $p$ is a valid jump function to define a bounded minimal speedup for a given odometer $(X,T)$.
	
	\begin{theorem}
		Let $(X,T)$ be an odometer with $\alpha = \langle \alpha_1,\alpha_2,\alpha_3,\ldots\rangle$.  Let $m_i = \alpha_1\alpha_2\cdots\alpha_i$ for all $i\in \Z^+$ and suppose $\{\mathcal{P}(k)\}$ is a nested sequence of partitions of $X$ labeled as in Theorem \ref{BKTHM}.  Then 
		$T \pleadsto S$ where $S:X \to X$ is minimal 
		if and only if there exists an $I\in \Z^+$ such that the following hold.
		\begin{enumerate}[(1)]
			\item For each $j=0,1,\ldots, m_I-1$, there exists $q_j\in\Z^+$ such that $p(x) = q_j$ for all $x\in T^jA(k) \in 
			\mathcal{P}(k)$.
			\item The elements of $\mathcal{P}(I)$ are cyclically permuted under $S:x \mapsto T^{p(x)}(x)$.
			\item $\displaystyle \sum_{j=0}^{m_I - 1} q_j = c\cdot m_I$ where $(c,\alpha_k) = 1$ for all $k>I$.
		\end{enumerate}
	\end{theorem}
	\begin{proof}
		First suppose $T \pleadsto S$ where $S:X \to X$ is minimal. 
		Then $p:X\to \Z^+$ is a continuous function and we may choose $I\in \Z^+$ such that (1) holds.  As $S$ is a minimal homeomorphism, (2) must also hold.  Further, because $S$ and $T$ cyclically permute the elements of $\mathcal{P}(I)$, we have $$S^{m_I}(T^jA(k)) = 
		T^{q_0+q_1+\cdots+q_{m_I-1}}(T^jA(k)) = T^jA(k) \text{ for all } j=0,1,\ldots,m_I-1$$ and 
		$$\displaystyle m \text{ divides } \sum_{k=0}^{m_I-1}q_k.$$ 
		Set $c\in Z^+$ such that $c\cdot m_I = \displaystyle\sum_{j=0}^{m_I-1} q_j$ and suppose that $(c,\alpha_k) \neq 1$ for some $k>I$. Without loss of generality we may assume $k=I+1$ and set $\displaystyle N = \frac{\alpha_{k}}{(c,\alpha_k)}$. Note that $$S^{m_I}(T^jA(k)) = T^{c\cdot m_I}(T^jA(k)) = T^jA(k) \text{ for all } j=0,1,\ldots, m_I-1.$$ Because $T^l A(I+1) \subset T^j A(I)$ whenever $l\equiv j\mod m_I$, $$S^{m_I}(T^l A(I+1)) = 
		T^{c\cdot m_I}(T^l A(I+1)) \text{ for all } l=0,1,\ldots m_{I+1}-1.$$  Then $$S^{N\cdot m_I}A(I+1) = T^{N\cdot c\cdot m_I} A(I+1) = A(I+1).$$  As $N\cdot m_I < m_{I+1}$, $S$ is not minimal, a contradiction.  Hence (3) must hold.
		
		Conversely, suppose that there exists an $I\in Z^+$ such that (1) - (3) hold.  We show that $S:x \mapsto T^{p(x)}(x)$ is a minimal bijection (and thus a minimal homeomorphism) on $X$.  First, suppose there exist $y,y'\in X$ such that $S(y) = S(y')=x$.  By (2) $y$ and $y'$ must lie in the same element $T^jA(I)$
		of $\mathcal{P}(I)$. By (1), $S(y) = T^{q_j}(y) = x = T^{q_j}(y') = S(y')$. As $T$ is one-to-one, it follows that $y=y'$. Further, given $x\in X$, there exists some $j$ such that 
		$x\in T^jA(I)$ and again by (2) there exists $T^l A(I)$ such that $S(T^l A(I))=T^jA(I)$ and $p(y) = q_l$ for all $y\in T^lA(I)$. As $T$ is onto, there exists a point $y\in T^jA(I)$ such that 
		$T^{q_l}(y)=S(y) = x$, and hence $S$ is an onto function. 
		The map $S$ is minimal by Lemmas \ref{permpower} and \ref{minsuff}. 
	\end{proof}

\section{Substitution Subshifts} \label{sub}
\subsection{Subshifts} \label{subshift}
Let $\mathcal{A}$ denote a finite set which we will refer to as an 
\emph{alphabet}, the elements of which we will refer to as \emph{symbols}. 
Let $\mathcal{A}^*$ denote the set of finite concatenations of symbols in $\mathcal{A}$
which we will refer to as \emph{words}. For a word $w=w_1w_2\cdots w_n$, we let $|w|=n$ denote the length of $w$. 

The set $\mathcal{A}^{\mathbb{Z}}$ is a Cantor space 
with the product of the discrete topology. We will consider the shift map
$T: \mathcal{A}^{\mathbb{Z}} \to \mathcal{A}^{\mathbb{Z}}$ given by 
$T(x)_k= x_{k+1}$. A \emph{subshift} is any closed, shift-invariant 
subset $X$ of such a space
along with the shift map $T$ restricted to $X$. 
It is a well-known theorem that $(X,T)$ is an \emph{expansive}
homeomorphism of a Cantor set if and only if 
$(X,T)$ is conjugate to a subshift, e.g. see \cite{LindMarcus}.

\begin{definition}
	A homeomorphism of a compact metric space $T:(X,d) \to (X,d)$ is 
	\emph{expansive} if there is a $\delta>0$ such that for every $x \neq y \in X$ there is an 
	$n \in \mathbb{Z}$ such that $d(T^nx,T^ny)>\delta$. 
\end{definition}

For $x \in \mathcal{A}^{\mathbb{Z}}$ and $i<j$, we will use the following notations:
$x[i,j]=x_ix_{i+1}\cdots x_j$, $x[i,j)=x_ix_{i+1}\cdots x_{j-1}$ and $x[i]=x_i$. If $w$ is a word in $\mathcal{A}^*$ we will use the same notation to denote subwords of 
$w=w_1w_2\cdots w_n$, e.g. 
$w[i,j] = w_iw_{i+1}\cdots w_j$ if $1\leq i<j\leq |w|$. 
The \emph{language} of a subshift $X$ is the set
of all words $\{x[i,j) :x \in X, i<j\}$. 
A subshift $(X,T)$ is minimal if and only if for every word $w$ in the language of $X$ there
is an $r>0$ such that for every $x \in X$ and every $i \in \mathbb{Z}$, 
$w$ is a subword of $x[i,i+r]$. 

\begin{example} \label{exammeas}
We interject here an example of a bounded speedup $T \pleadsto S$ of a subshift $(X,T)$ 
with orbit number $2$ such that $p(x)-2$ has a non-zero integral for some $S$-invariant Borel 
probability measure. This will follow from the existence of a point
$x_0 \in X$ such that 
$$\limsup_{N \to \infty} 
\frac{1}{N} \sum_{j=0}^{N-1} p(S^jx_0) > 0.$$

First define the space $X \subset \{0,1\}^{\mathbb{Z}}$ upon which
the shift map $T$ acts. 
The increasing sequence of integers $6 < n_2 < n_3 < \cdots $ 
will be recursively defined later.
Set 
$$w_1(0) = 0 0 0 0 0 1 \quad \text{and} \quad w_1(1) = 0 0 0 0 0 1 1$$
and for $k \geq 2$, 
$$w_{k}(0) = w_{k-1}(0)^{n_k+1} w_{k-1}(1)w_{k-1}(0)w_{k-1}(1)$$
and 
$$w_{k}(1) = w_{k-1}(0)^{n_k+1} w_{k-1}(1)w_{k-1}(0)w_{k-1}(1)^2.$$
We define $X$ by saying that the language of $X$ is the set of all words that are 
subwords of $w_k(0)$ or $w_k(1)$ for some $k \geq 1$. One can check that the
system $(X,T)$ is minimal. 

Set $A = \{ x : x[0,6) = 000001 \} \subset X$ and 
define a jump function $p:X \to \mathbb{Z}^+$ as follows. 
Let$$p(x)=\begin{cases}
4 & \text{ if $x\in A$}\\
1 & \text{ if $x \in T^j A$ for $j =1,2$}\\
2 & \text{ otherwise.}
\end{cases}$$

Consider $x_0 \in X$ with $x_0[0,|w_k(0)|) = w_k(0)$ for all 
$k \geq 1$ and let 
$s_k = \min \{N : \sum_{j=0}^{N-1} p(S^jx_0) \geq |w_k(0)|\}$. 

We see 
\begin{align*}
\sum_{j=0}^{s_{1}-1} p(S^jx_0) & = 6 \\
\sum_{j=0}^{s_{k}-1} p(S^jx_0) & = n_k \sum_{j=0}^{s_{k-1}-1} p(S^jx_0) + 2|w_{k-1}(0)|+2|w_{k-1}(1)|.
\end{align*}

In order to complete the example, we need the following recursive formula as well. 
\begin{align*}
s_1 & = 2 \\
s_{k} & = n_k s_{k-1} + |w_{k-1}(0)| + |w_{k-1}(1)|
\end{align*}
The recursion formulae show that by choosing $n_k$ sufficiently large, 
we can make 
$\displaystyle \frac{1}{s_k}  \sum_{j=0}^{s_{k}-1} p(S^jx_0)$ as close to  
$\displaystyle \frac{1}{s_{k-1}}  \sum_{j=0}^{s_{k-1}-1} p(S^jx_0)$ as we like. 

We see $\displaystyle \frac{1}{s_1}  \sum_{j=0}^{s_{1}-1} p(S^jx_0) = 3$. 
Therefore, we may recursively choose $n_2 < n_3 < \cdots$ so that 
$\displaystyle \frac{1}{s_{k}}  \sum_{j=0}^{s_{k}-1} p(S^jx_0) >2$ for all $k \geq 1$. 

\end{example}

If $(X,T)$ is a subshift with alphabet $\mathcal{A}$, 
the \emph{$m$-block presentation of $(X,T)$}
is the shift map acting on the space $X^{[m]}$ where the alphabet is $\mathcal{A}^m$
and a sequence $w = (w_i)_{i \in \mathbb{Z}}$ is in $X^{[m]}$ if and only if 
\begin{enumerate}
	\item $w_i[1,m)=w_{i+1}[0,m-1)$ for all $i$,
	\item the sequence $(w_i[0])_{i \in \mathbb{Z})}$ is in $X$. 
\end{enumerate}
For any $m \geq 1$, the $m$-block presentation of a subshift is conjugate to the 
subshift itself. 
Given a function $f \in C(X,\mathbb{Z})$, we will use 
higher block presentation to assume without loss of generality that 
$f(x)$ depends only on the symbol $x[0]$. 

\begin{lemma} \label{lem:exp}
	A bounded speedup of an expansive map is expansive. 
\end{lemma}
\begin{proof}
	Let $(X,T)$ be expansive, i.e., a minimal subshift, and 
	suppose $T \pleadsto S$. We may assume that the jump function 
	$p(x)$ depends only on the value of $x_0$. Then $(X,S)$ is 
	conjugate to the subshift with symbols $\mathcal{B}= \{ x[0,p(x)) : x \in X\}$ 
	and where a sequence in these symbols $(w_i)$ is allowed if and only if
	for all $i\in \mathbb{Z}$ and $r>0$ the word concatention 
	$w_iw_{i+1}\cdots w_{i+r}$ is equal to $x[0,m]$ for some 
	$x \in X$ and some $m>0$.
\end{proof}

The following will allow us to show that 
a bounded speedup of a subshift is never conjugate to the original system, 
except possibly when the original system is a periodic action on a finite set.

\begin{lemma} \label{lem:word}
	Let $(X,T)$ be a subshift and let 
	$W_n(X)$ denote the set of words of length $n$ appearing in $X$. 
	If $|W_n(X)| \geq |W_{n+j}(X)|$ for some $n,j > 0$, then $X$ is finite. 
\end{lemma}
\begin{proof}
	Assume $|W_n(X)| \geq |W_{n+j}(X)|$. Let $\pi : W_{n+j}(X) \to W_n(X)$ denote 
	projection onto the first $n$ letters, 
	$\pi(w_1w_2 \cdots w_{n+j}) = w_1w_2 \cdots w_n$. Since $\pi$ is an onto function, 
	we see $|W_n(X)| = |W_{n+j}(X)|$ and $\pi$ is a bijection. 
	
	Since $\pi$ is a bijection, for every $x \in X$, there is a unique
	word $w \in W_{n+j}(X)$ such that $x[0,n) = \pi(w)$. In other words, 
	there is a unique word $b \in W_j(X)$ such that $x_0x_1\cdots x_{n-1}b$ is in the language of $X$. 
	This means that $x[0,n)$ determines $x[0,n+j)$. Likewise the word $x[j,n+j)$ determines $x[j,n+2j)$, etc. 
	Therefore $x[0,n)$ determines the right infinite word 
	$x_0x_1x_2\cdots$. 
	
	Repeating the argument with $\pi$ replaced by projection onto the \emph{last} $n$ letters, 
	we see that $x[0,n)$ determines the left infinite word 
	$\cdots x_{-2}x_{-1}x_0$ as well. Therefore, $|X| \leq |W_n(X)|$, and in particular, $X$ is finite. 
\end{proof}

\begin{lemma}
	Let $(X,T)$ be a minimal Cantor system. Let $p:X \to \mathbb{Z}^+$ be a continuous function such that 
	$p$ is not the constant function $1$ and let $m$ be given. There is an $N$ such that for any point $x \in X$, 
	$$\sum_{j=0}^{N-1} p(T^jx) > N+m$$
\end{lemma}

\begin{proof}
	Because $p(x) \geq 1$ for all $x \in X$ and $p \not \equiv 1$, there is a clopen set $U$ such that 
	for all $x \in U$, $p(x) \geq 2$. Since $T$ is minimal there is an $r$ such that for any $x \in X$, 
	one of the points $x,Tx,T^2x, \ldots T^{r-1}x$ is in $U$. In other words, 
	$$\sum_{j=0}^{r-1} \left( p(T^jx) -1 \right) \geq 1$$
	for all $x \in X$. 
	
	Choose $N > (m+1)r$ and let $x \in X$. 
	Then 
	$$\sum_{j=0}^{N-1} \left(p(T^jx) -1 \right)  \geq \sum_{k=0}^{m} \sum_{j=0}^{r-1} \left(p(T^{kr+j}x)-1\right)  \geq m+1$$
	Rearranging, we get $\sum_{j=0}^{N-1} p(T^jx) > N+m$
\end{proof}

\begin{theorem}
	Suppose $(X,T)$ is a minimal subshift where $X$ is infinite and
	$T \pleadsto S$. Then $(X,T)$ is not conjugate to $(X,S)$. 
\end{theorem}
\begin{proof}
	Without loss of generality, $p(x)$ only depends only on $x[0]$, the zero-th coordinate of $x$. 	Now assume that $(X,T)$ is conjugate to $(Y,S)$ where $(Y,S)$ is the subshift defined in Lemma \ref{lem:exp} via a conjugacy $\phi: X \to Y$. The map 
	$\phi$ is a sliding block code \cite{LindMarcus}. In other words
	$\phi$ is defined by a map 
	$\Phi:W_n(X) \to W_1(Y)$ for some $n \geq 0$. 
	Note that $W_1(Y)$ consists precisely of 
	words of the form $x[0,p(x))$ in $X$.

	For every $N>0$, we may extend $\Phi$ by concatenation to obtain an onto function 
	$\Phi:W_{n+N}(X) \to W_{N+1}(Y)$. Elements of 
	$W_{N+1}(Y)$ naturally project to words of the form 
	$x\left[0,\sum_{j=0}^{N} p(S^j x) \right)$ in $X$. 
	Applying the previous lemma, if $N$ is sufficiently large, we can guarantee that 
	for all $x$, $\sum_{j=0}^{N} p(S^j x) > n+N$. 
	By this inequality every word of length $n+N$ in $X$ 
	is a subword of $x\left[0,\sum_{j=0}^{N} p(S^j x) \right)$ for some $x$. 
	Putting this all together, we obtain an onto function from $W_{n+N}(X)$ to $W_{n+N+1}(X)$.
	Thus $|W_{n+N}(X)| \geq |W_{n+N+1}(X)|$ and $X$ is finite. 
\end{proof}

\subsection{Substitution Subshifts}

Here we consider subshifts generated by a substitution map $\theta$. 
Let $\mathcal{A}^*$ denote the set of finite concatenations of symbols in an alphabet 
$\mathcal{A}$, 
and let $\theta:\mathcal{A} \to \mathcal{A}^*$ be a function which we call 
a \emph{substitution function}.
We may extend $\theta$ to a map from $\mathcal{A}^* \to \mathcal{A}^*$ by 
concatenation and in so doing consider iterations 
$\theta^k:\mathcal{A}^* \to \mathcal{A}^*$. 
Given such a map $\theta$, we may consider the subshift
$X^{\theta}\subset \mathcal{A}^{\mathbb{Z}}$, 
the set of all $x \in \mathcal{A}^{\mathbb{Z}}$ such that for all $i< j$ in 
$\mathbb{Z}$, $x[i,j]$ is a subword of $\theta^k(a)$ for some $k \geq 0$ and
some $a \in \mathcal{A}$. 

We require additional properties of $\theta$ in order to insure that $X^{\theta}$ is a
minimal Cantor system. 

\begin{definition}
	Let $\theta: \mathcal{A} \to \mathcal{A}^*$ be a substitution function. We
	say that $\theta$ is \emph{primitive} if 
	\begin{enumerate}
		\item for any $a,b \in \mathcal{A}$, there is a $k \geq 0$ such that 
		$b \in \theta^k(a)$
		\item for any $a \in \mathcal{A}$, $\lim_{k \to \infty}|\theta^k(a)|=\infty$.
	\end{enumerate}
\end{definition}

\begin{definition}
	Let $\theta: \mathcal{A} \to \mathcal{A}^*$ be a substitution function. We
	say that $\theta$ is \emph{proper} if there exists 
	$\ell, r \in \mathcal{A}$ and a $k \geq 0$ such that 
	for all $a \in \mathcal{A}$, every word $\theta^k(a)$ begins with the symbol $\ell$
	and ends with the symbol $r$. 
\end{definition}

\begin{definition}
	Let $\theta: \mathcal{A} \to \mathcal{A}^*$ be a substitution function. We
	say that $\theta$ is \emph{aperiodic} if the subshift $X^{\theta}$ contains 
	no periodic points. 
\end{definition}

Recall the following theorems about primitive, proper, aperiodic substitutions. We refer the reader to
\cite{DHS} for more details on these results. 

\begin{theorem}
	If $\theta$ is a proper, primitive, aperiodic substitution, then 
	$(X^{\theta},T)$ is a minimal Cantor system. 
\end{theorem}

\begin{theorem}
	If $(X^{\theta},T)$ is a minimal Cantor system associated to a
	substitution $\theta$, then there is a proper, primitive, aperiodic
	substitution $\tau$ such that $(X^{\theta},T)$ and $(X^{\tau},T)$
	are topologically conjugate. 
\end{theorem}

Equipped with these preliminaries, we will proceed by examining examples of bounded speedups of substitutions which will help illuminate our general results. 

\subsection{Powers versus bounded speedups}

In our first example we show that there can be a bounded minimal speedup $T \rightsquigarrow S$ with orbit number $2$ even when $T^2$ is not minimal (see also remark \ref{remex}). 
In so doing, we show that 
studying bounded speedups of substitutions is more general than studying powers of substitutions. Because $T^2$ is not minimal, but $S$ is minimal,  
$S$ cannot be conjugate to $T^2$. It follows from Lemma \ref{lem:word} that $T^k$ is not conjugate to $S$ for $k > 2$. 

\subsubsection{Example} \label{ex-orbnum2}
Consider the substitution below on $\mathcal{A}=\{0,1\}$. 
\begin{center}
	\hspace*{\fill}
	$\theta:0\mapsto 0011\hfill \theta: 1\mapsto 001011$
	\hspace*{\fill}
\end{center}
This is a primitive, proper, aperiodic substitution and 
$(X^{\theta},T)$ is
a minimal Cantor system with respect to the shift map $T$. 
Note here that since the $\theta$-word lengths are all even, 
$T^2$ is not minimal. 

On the other hand, consider $S(x) = T^{p(x)}(x)$ where 
$p:X\rightarrow\Z^{+}$ is defined below. 
Set $A = \{x: x[0,5] = 001011 \}$ and let 
$$p(x)=\begin{cases}
3 & \text{ if $x\in A$}\\
1 & \text{ if $x\in TA$}\\
2 & \text{ otherwise.}
\end{cases}$$
Letting $g(x)$ be the indicator function of $A$, we see that 
$$p(x) = 2 + g(x) - g(T^{-1}x).$$
Below we develop some general theory to show that $S(x)$ is minimal which will also be useful 
in following sections. 

\subsubsection{Kakutani-Rokhlin  Partitions and Substitutions}
We will first introduce Kakutani-Rokhlin partitions as they relate to substitutions. If $\theta$ is a primitive, 
aperiodic substitution, then for every $x \in X$ and every 
$k \geq 1$, there is a decomposition of the sequence 
$x$ into $\theta^k$-words; it follows from Theorems of 
Moss\'{e} that this decomposition is unique \cite{Mosse1,Mosse2}. In other words,
for every $x \in X$ and every $k \geq 1$ 
there exist a unique set of integers
$\cdots < n_{-2} < n_{-1} < 0 \leq n_0 < n_1<n_2<\cdots $ 
and symbols $\{a_j \in \mathcal{A} :  j \in \mathbb{Z}\}$ 
such that for all $j$, $x[n_j,n_{j+1}) = \theta^k(a_j)$.

It will ease our notation to assume the substitution $\theta$ is defined on 
$I = \{1,2,\ldots, |\mathcal{A}|\}$. For each $i\in I$  
let $A_i(k)$ denote the set of points $x\in X$ such that in the decomposition of $x$ 
into $\theta^k$-words, $n_0=0$ and
$x[n_0,n_1)=\theta^k(i)$. For $i \in I$, let $l_i(k) = |\theta^k(i)|$.
Let $\mathcal{P}(k)=
\{ T^j A_i(k) \ : \ i \in I, 0 \leq j < l_i(k)\}$.
\begin{proposition}
	Suppose $\theta$ is a proper, primitive, aperiodic substitution on an alphabet $\mathcal{A}$. 
	With the above notation, 
	\begin{itemize}
		\item each $\mathcal{P}(k)$ is a clopen partition of $X^{\theta}$,
		\item the partitions $\{ \mathcal{P}(k) : k \geq 1 \}$ generate the topology of $X^{\theta}$, 
		\item the set $\cap_{k \geq 1} \cup_{i \in I} A_i(k)$ is a singleton. 
	\end{itemize}
\end{proposition}

\subsubsection{Orbit Block Labeling}
As in the previous section, given a bounded speedup $T \pleadsto S$, 
for sufficiently large $k$, 
we introduce a \emph{labeling} $\mathcal{L}_k$ of $\mathcal{P}(k)$ based on $S$-paths through each column. By selecting $k$ sufficiently large and using the fact $\theta$ is proper, we may assume 
\begin{enumerate}
	\item \label{jumpcont}
	the jump function $p$ is constant on floors of $\mathcal{P}(k)$, 
	\item  \label{tallcol}
	$l_i(k) > \max p$ for all $i$,
	\item \label{lowfloors} the jump function $p$ is constant on 		
	sets of the form 
	$T^n \left( \cup_i A_i(k) \right)$ for $0 \leq n \leq \max p$. 
\end{enumerate}

Our labeling will be a function $\mathcal{L}_k$ from $\mathcal{P}(k)$ to the 
set $\{0,1,\ldots, c-1\}$ where $c$ is the orbit number for $T \rightsquigarrow S$. 
Fix $i$, and consider the floors of the $i$th column of $\mathcal{P}(k)$. 
We define the labeling recursively beginning with the base floor. 

Label the base floor $A_i(k)$ with $0$. 
Label any other floor $F$ in this column with a $0$ 
if $F = S^jA_i(k)$ and $\sum_{k=0}^{j-1} pS^k(x) < l_i(k)$ for $x \in A_i(k)$. 
Now consider the lowest unlabeled floor $T^{j_1}A_i(k)$. Label this floor with a $1$. 
Label any other floor $F$ in this column with a $1$ if $F = S^jA_i(k)$ 
and $\sum_{l=0}^{j-1} pS^l(x) < l_i(k)$ for $x \in T^{j_1} A_i(k)$.
Label the lowest unlabeled floor with a $2$ and continue. 
Continue in this manner until all floors in the $i$th column are labeled with a label
$0,1,\ldots , c-1$. 

The labeling has the property that if a floor $F$ is the 
$S$-image of a lower floor $E$ in the same column, then 
$E$ and $F$ have the same labeling.
It is also the case that if $F$ is a floor of $\mathcal{P}(k)$ labeled $\ell$ 
and is not the floor of maximal height in column $i$ with this property then 
$SF$ is equal to the next higher floor in column $i$ which has label $\ell$. 

Note that item \ref{lowfloors} above guarantees that the labels on any two floors with 
the same height $<\max{p}$ are the same. Therefore, if $F$ is a floor of maximal height in 
column $i$ of $\mathcal{P}(k)$ with label $\ell$ then for all $x,y \in F$ 
then the label of the floor containing
$S(x)$ is the same as the label of the floor containing $S(y)$. 

\subsubsection{Orbit Block Labeling Permutations}
For $k$ large enough so that conditions \ref{jumpcont}, \ref{tallcol} and \ref{lowfloors}
above, define $\pi^{(k)}_i(\ell)$ to be the label of the floor containing $S(x)$ for all 
$x$ in the floor of maximal height in column $i$ labeled $\ell$.
\begin{proposition}
	Each $\pi^{(k)}_i$ is a permutation of the set $\{0,1, \ldots ,c-1\}$.
\end{proposition}
\begin{proof}
	This follows from the fact that $\pi^{(k)}_i$ is injective. If $\pi^{(k)}_i$ 
	is not injective then there are two distinct points in the same $T$-orbit 
	which have the same $S$-image. 
\end{proof}

Let $\pi^{(k)}$ denote the tuple of permutations 
$\langle \pi^{(k)}_1, \pi^{(k)}_2, \ldots \pi^{(k)}_n \rangle$.

\begin{lemma} \label{constperm}
	There exists a $K \geq 1$ such that $\pi_i^{(K)}=\pi_i^{(jK)}$ for all $1 \leq i \leq n$
	and all $j \in \mathbb{N}$. 
\end{lemma}

\begin{proof}
	There is a well-defined function that transforms the vector of 
	permutations $\langle \pi_1^{(k)}, \pi_2^{(k)}, \ldots, \pi_n^{(k)} \rangle$ to 
	$\langle \pi_1^{(k+1)}, \pi_2^{(k+1)}, \ldots, \pi_n^{(k+1)} \rangle$ given by the following:
	if $\theta (a_i) = a_{i_1}a_{i_2}\cdots a_{i_m}$ then 
	$\pi_i^{(k+1)} = \pi_{i_m}^{(k)} \circ \pi_{i_{m-1}}^{(k)}\circ \cdots \circ \pi_{i_1}^{(k)}$. 
	Set $\pi^{(k)} = \langle \pi_1^{(k)}, \pi_2^{(k)}, \ldots, \pi_n^{(k)} \rangle$. 
	Because there are only finitely many possibilities for $\pi^{(k)}$ this 
	transformation is eventually periodic with some period $N$. If $K$ is a sufficiently large
	multiple of $N$ then $\pi^{(K)}= \pi^{(jK)}$ for all $j \in \mathbb{N}$.
\end{proof}

\subsubsection{A sufficient condition for minimality of $S$}
Consider $k$ satisfying the conditions of Lemma \ref{constperm}
and sufficiently large to satisfy 
conditions \ref{jumpcont}, \ref{tallcol} and \ref{lowfloors}.
The minimal system that is generated by the substitution 
$\theta^k$ is the same as that generated by $\theta$. Thus we
may assume without loss of generality that $\theta$ satisfies all 
of these hypotheses with $k=1$ and Lemma \ref{constperm} with $K=1$. 
As such all $\pi^{(k)}$ are equal and 
we will generally drop the superscript going forward. 

Given $x \in X$, then $x$ has a unique decomposition into $\theta$-words, 
$x[n_j,n_{j+1})=\theta(i_j)$, $j \in \mathbb{Z}$. Associated to this decomposition is a
label sequence $(\ell_j)_{j \in \mathbb{Z}}$ in the following way. 
The point $x$ is an element of one of the floors of $\mathcal{P}(1)$, let 
$\ell_0$ be the $\mathcal{L}_1$-label for this floor. 
For $j\geq 0$ set $\ell_{j+1} = \pi_{i_j}(\ell_j)$ and 
$j \leq 0$ set $\ell_{j-1} = \pi_{i_{j-1}}^{-1}(\ell_j)$. In this way, 
if $S^l(x) = T^n(x)$ where $n \in [n_j,n_{j+1})$ then 
$S^l(x)$ belongs to a floor with $\mathcal{L}_1$-label $\ell_j$ in the 
$i_j$ column of $\mathcal{P}(k)$. 
Thus to every $x \in X$, we can associate the sequence
$(i_j, \ell_j)$, which we will refer to as the 
\emph{symbol-label} sequence for $x$.

Now set $C = \{0,1,\ldots, c-1\}$ and consider an alphabet $I \times C$. 
Let $\sigma:I \times C \to (I \times C)^*$ be the function 
defined by $$\sigma(i,\ell) = (i_1,\ell_1)(i_2,\ell_2)\cdots(i_m,\ell_m)$$
where $\theta(i) = i_1 i_2 \cdots i_m$, $\ell_1 = \ell$ and 
$\ell_{k+1}=\pi_{i_k}(\ell_k)$ for $k\geq 1$. 

Note that the language generated by $\sigma$ is precisely the set of 
	symbol-label sequences for $x \in X$. 
Note further that the exact same analysis applies to $\theta^k$ and labeling $\mathcal{L}_k$.
Since $\pi^{(k)}$ is the same for all $k$, we have the same 
substitution map $\sigma$ generating the same symbol-label sequences.  

\begin{lemma}
	Suppose $\theta$ is a proper, primitive, aperiodic substitution and 
	$(X,T)$ is the subshift generated by $\theta$. Further assume
	$T \pleadsto S$ where $S:X \to X$ is a homeomorphism. 
	Let $\sigma$ be the substitution defined as above,
	then $\sigma$ is primitive if and only if $S$ is minimal. 
\end{lemma}

\begin{proof}
	Suppose $\sigma$ is primitive and let $x$ be a point in $X$. 
	There is a power $r$ such that for all pairs of symbols $(i_j, \ell_j), (i_j', \ell_j')$
	the symbol $(i_j', \ell_j')$ appears in the word $\sigma^r(i_j,\ell_j)$. 
	Consider the decomposition into $\theta^{k+r}$ words. 
	Within each $\theta^{k+r}$ word are all possible
	$\theta^{k}$ words with all possible labels. Thus the $S$-orbit of $x$ intersects all 
	floors of $\mathcal{P}(k)$. Because this is true for all $k$, the $S$-orbit of $x$ is dense. 
	
	Suppose $S$ is minimal. There is an $N$ such that any $S$-orbit block 
	$\mathcal{O}(S,x,N+1)$ intersects all floors of $\mathcal{P}(1)$ with all labels. 
	Set $M = \sup_{x \in X} p(x)$. 
	Let $r$ be an integer such that $|\sigma^r(i,\ell)|>MN$ for all $(i,\ell)$. Then for all 
	$(i',\ell')$, $(i',\ell')$ must appear in $\sigma^r(i,\ell)$.
\end{proof}

Let us return now to Example \ref{ex-orbnum2}. Here one can check that 
$\pi_0^{(k)} = id$ and $\pi_1^{(k)}$ is the permutation $0 \leftrightarrow 1$ for 
all $k \geq 1$. Thus the substitution $\sigma$ in this case is given by 
\begin{align*}
\sigma:(0,0) & \mapsto (0,0)(0,0)(1,0)(1,1) \\
\sigma:(0,1) & \mapsto (0,1)(0,1)(1,1)(1,0) \\
\sigma:(1,0) & \mapsto (0,0)(0,0)(1,0)(0,1)(1,1)(1,0)\\
\sigma:(1,1) & \mapsto (0,1)(0,1)(1,1)(0,0)(1,0)(1,1).
\end{align*}
Since $\sigma(1,0)$ contains all four symbols and all $\sigma$-words contain the 
symbol $(1,0)$, this substitution is primitive and therefore the speedup $S$ is minimal. 

\subsection{$T$-coboundaries vs. $S$-coboundaries}
Using the above definitions and notation, we are able to give an explicit example of a substitution minimal Cantor system $(X,T)$ and a bounded speedup $T \pleadsto S$ such that
$p(x)-2$ is a $T$-coboundary but not an $S$-coboundary. 

Consider the following substitution example on $\mathcal{A}=\{0,1\}$
\begin{center}
	\hspace*{\fill}
	$\theta:0\mapsto 00011\hfill \theta: 1\mapsto 001$.
	\hspace*{\fill}
\end{center}
This is a primitive, proper, aperiodic substitution and $(X^{\theta},T)$ 
a minimal Cantor system with respect to the shift map $T$. 
Set $A = \{x : x[0,4]=00011 \}$. 
Define $p:X\rightarrow\Z^{+}$ as follows
$$p(x)=\begin{cases}
3 & \text{ if $x \in A$}\\
1 & \text{ if $x \in TA$}\\
2 & \text{ otherwise}
\end{cases}$$
Once again in this case, we see that 
$$p(x) = 2 + g(x) - g(T^{-1}x)$$
where $g(x)$ is the indicator function of $A$. 
Morevoer, $\pi_0^{(k)}=id$ 
and $\pi_1^{(k)}$ is the permutation $0 \leftrightarrow 1$ for 
all $k \geq 1$. The associated substitution $\sigma$ is given 
by 
\begin{align*}
\sigma:(0,0) & \mapsto (0,0)(0,0)(0,0)(1,0)(1,1) \\
\sigma:(0,1) & \mapsto (0,1)(0,1)(0,1)(1,1)(1,0) \\
\sigma:(1,0) & \mapsto (0,0)(0,0)(1,0)\\
\sigma:(1,1) & \mapsto (0,1)(0,1)(1,1).
\end{align*}
One can again check that $\sigma$ is primitive so 
$S:x \mapsto T^{p(x)}(x)$ is minimal. 

We will show that $f(x)=p(x)-2$ is not an $S$-coboundary, i.e., that 
$f(x)$ is not of the form $h(x)-h(Sx)$ where $h \in C(X,\mathbb{Z})$. 
We do so by showing that
$$
\sup_{n}\sum_{i=0}^{n-1}f(S^{i}z)=\infty.
$$
where $z$ is the fixed point of the substitution $\theta$. 
In other words, if for all $k \geq 1$, in the decomposition of $z$ into $\theta^k$-words, 
$n_0=0$ and $z[0,n_1)=\theta^k(0)$.  

Consider the integer sequence $\{l_{k}\}$ where $l_{k}$ is the number of $S$-steps it takes for $z$ to traverse the first $\theta^{k}(0)$-block. 
Formally, $l_{k}$ is the minimum integer such that
$$
\sum_{i=0}^{l_{k}-1}p(S^{i}z)\ge|\theta^{k}(0)|.
$$
One can check: $l_{1}=2,\, l_{2}=9,\, l_{3}=40$. Further, one can check:
\begin{align*}
\sum_{i=0}^{l_{1}-1}f(S^{i}x)&=1\\
\sum_{i=0}^{l_{2}-1}f(S^{i}x)&=3\\
\sum_{i=0}^{l_{3}-1}f(S^{i}x)&=9\\
\end{align*}
Inductively, we would like to see that $\sum_{i=0}^{l_{k}-1}f(S^{i}x)=3^{k-1}$, which would prove the claim. This follows because in order to traverse a $\theta^{k+1}(0)-$block, the point $z$ traverses three $\theta^{k}(0)-$blocks and two $\theta^{k}(1)-$blocks. 
Each of the $\theta^{k}(0)-$blocks contributes $3^{k-1}$ and the $\theta^{k}(1)-$blocks contribute nothing. 

The example above is particularly relevant as it demonstrates an example where 
$p(x) - 2$ is a $T$-coboundary but not an $S$-coboundary. If $p(x)-c$ is an 
$S$-coboundary of the form $h(x)-h(Sx)$, then in fact $T^c$ and $S$ are conjugate via the map 
$T^{h(\cdot)}$. We can further see that in the above example $T^2$ and $S$ are not conjugate
by any map as they have different associated dimension groups. 

\subsection{A minimal 
bounded speedup of a substitution is a substitution}

We will use in this section a characterization of minimal Cantor substitution systems as \emph{expansive} and \emph{self-induced}.

Let $(X,T)$ be a minimal Cantor system and let $A$ be a proper clopen subset 
of $X$. Due to the minimality of $T$ for every $x \in A$ there is a first
return time $r(x) = \min \{n>0 : T^n(x) \in A\}$. We may then consider 
the \emph{induced map} $T_A:A \to A$ defined by $T_A(x) = T^{r(x)}(x)$. 
As it turns out, the induced system $(A,T_A)$ is also a minimal Cantor system. 

\begin{definition}
	We say that a minimal Cantor system $(X,T)$ is \emph{self-induced} if 
	$(X,T)$ is conjugate to $(A,T_A)$ where $A$ is a proper clopen subset of $X$ and
	$T_A:A \to A$ is the induced map on $A$. 
\end{definition}

It is not difficult to see that a minimal substitution system $(X^{\theta},T)$
is self-induced via the extension of the map $\theta$ to sequences in $X^{\theta}$. 
In \cite{DOP}, the converse is proven. 

\begin{theorem}
	Let $(X,T)$ be an expansive, self-induced minimal Cantor system. Then 
	$(X,T)$ is conjugate to a substitution system $(X^{\theta},T)$ where
	$\theta$ is a primitive, aperiodic, proper substitution. 
\end{theorem}

We will use the theorem above to show the main theorem in this section, that
a bounded speedup of a substitution system is conjugate to a substitution system. 

\begin{theorem} \label{subspeed}
Suppose $(X^{\theta}, T)$ is a minimal substitution system associated with the proper, primitive substitution $\theta$.  
If $T \pleadsto S$ where $S:X^{\theta} \to X^{\theta}$ is minimal, 
then $(X^{\theta}, S)$ is a substitution system.
\end{theorem}

It remains to show that a bounded speedup of a substitution system is 
self-induced. That is, we will show that $(X^{\theta},S)$ is topologically 
conjugate to $(U,S_U)$ where $U \subset X^{\theta}$ is
a clopen subset of $X^{\theta}$. 

\begin{theorem}
	Suppose $(X^{\theta},T)$ is a substitution system given by a proper, primitive substitution $\theta$
and that $T \pleadsto S$ where
$S:X^{\theta} \to X^{\theta}$ is minimal. Then $(X^{\theta},S)$ is self-induced. 
\end{theorem}
\begin{proof}
	We will define a map  $\varphi: \mathcal{P}(1) \to \mathcal{P}(2)$ and then 
	let $U$ be the union of the elements of $\mathcal{P}(2)$ which are in the range of
	$\varphi$.
	
	Fix $i$ and consider the $i$th column of $\mathcal{P}(2)$, $\{T^j A_i(2) : 0 \leq j < l_i(2)\}$. 
	Let $B(1) = \cup A_i(1)$, the base of the tower partition $\mathcal{P}(1)$. 
	Among these sets in the $i$th column of $\mathcal{P}(2)$, there are precisely 
	$l_i(1)$ which are a subset of $B(1)$. Let 
	$0 < r_1 < r_2 < \cdots < r_{l_i(1)-1}$ be the heights of these floors. 
	
	Set $\varphi(A_i(1))=A_i(2)$. For $j \in [1,l_i(1))$, 
	set $\varphi(T^j A_i(1))$ to be a floor $T^k A_i(2)$ with 
	$k \in [r_j,r_{j+1})$ 
	which has an $\mathcal{L}_2$-label equal to the $\mathcal{L}_1$-label of $T^j A_i(1)$. 
	
	We can extend $\varphi$ to a map on points using the expansiveness of $T$, which
	we also denote by $\varphi$. Let us check that $\varphi S(x)=S_U \varphi(x)$.
	There are two cases depending upon the element $T^jA_i$ of 
	$\mathcal{P}(1)$ that contains $x$. Suppose 
	$x \in T^j A_i(1)$; then either $p(x) + j < l_i(1)$ or otherwise. 
	
	\noindent \underline{$p(x) + j < l_i(1)$}
	Let $F$ denote the floor of $\mathcal{P}(1)$ containing $x$.
	In this case, $SF$ is the next higher floor $E$ in the $i$th column with the same 
	$\mathcal{L}_1$ label as $F$. The floors $\varphi(F)$ and $\varphi(E)$ have the same 
	$\mathcal{L}_2$ labels, and no images $\varphi$ between them have this label. 
	Therefore, $S_U\varphi(F)=\varphi(E)=\varphi S(F)$ and for all $x \in F$, 
	$S_U\varphi(x)=\varphi S(x)$.
	
	\noindent \underline{$p(x) + j \geq l_i(1)$}
	Let $F$ denote the floor of $\mathcal{P}(1)$ containing $x$.
	In this case, if $F$ has label $\ell$ then it is the highest level in the 
	$i$th column of $\mathcal{P}(1)$ with this label. 
	Thus $\varphi(F)$ is the highest level in the $i$th column of
	$\mathcal{P}(2)$ which is a subset of $U$ and is labeled $\ell$. 
	
	Let $x \in F$, with $S(x) \in E$ in the $m$th column of $\mathcal{P}(1)$. 
	The floor $E$ has label $\pi_i(\ell)$ and is the lowest floor in its column
	of $\mathcal{P}(1)$ with this label. The next $S$-entry of $\varphi(x)$ into $U$ 
	is in the first floor from the $m$th column of $\mathcal{P}(2)$ with $\mathcal{L}_2$-label
	$\pi_i(\ell)$, and therefore this is
	the set $E$. Thus $\varphi S (x) = S_U \varphi(x)$. 
\end{proof}

\end{document}